\documentclass{svjour3}

\usepackage{dsfont}
\usepackage{amsfonts,amssymb,amsmath}
\usepackage{epsfig}
\usepackage{cite}
\usepackage[weather,alpine,misc,geometry]{ifsym}
\smartqed

\DeclareMathOperator*{\liminflarge}{\mbox{\large\rm liminf}}
\DeclareMathOperator*{\inflarge}{\mbox{\large\rm inf}}

\newcommand{\norm}[1]{\left\Vert#1\right\Vert}
\newcommand{\TO}[1]{\stackrel{#1}{\to}}

\newcommand {\R} {\mathbb R}

\newcommand {\B} {\mathbb B}
\newcommand {\sd} {\partial}
\newcommand {\ptrog} {{\rm int}\,}
\newcommand {\gr} {{\rm gph}\,}

\def\RHS{right-hand side}

\spnewtheorem{thm}{Theorem}[section]{\bfseries}{\rmfamily}

\spnewtheorem{prop}{Proposition}[section]{\bfseries}{\rmfamily}

\spnewtheorem{lem}{Lemma}[section]{\bfseries}{\rmfamily}

\spnewtheorem{rem}{Remark}[section]{\itshape}{\rmfamily}

\spnewtheorem{cor}{Corollary}[section]{\bfseries}{\rmfamily}

\spnewtheorem{examp}{Example}[section]{\itshape}{\rmfamily}

\spnewtheorem{dfn}{Definition}[section]{\bfseries}{\rmfamily}

\begin{document}
\title{Quantitative Characterizations of Regularity Properties of Collections of Sets
\thanks{The research was supported by the Australian Research Council, project DP110102011.
}}

\author{Alexander~Y.~Kruger \and Nguyen~H.~Thao}

\institute{A. Y.~Kruger (\Letter\,) \at
Centre for Informatics and Applied Optimization, School of Science, Information Technology and Engineering, Federation University Australia, POB 663, Ballarat, Vic, 3350, Australia\\
\email{a.kruger@federation.edu.au} \and N. H.~Thao \at
Centre for Informatics and Applied Optimization, School of Science, Information Technology and Engineering, Federation University Australia, POB 663, Ballarat, Vic, 3350, Australia\\
\email{hieuthaonguyen@students.federation.edu.au, nhthao@ctu.edu.vn}}

\date{Received: date / Accepted: date}

\journalname{J Optim Theory Appl}
\maketitle

\begin{abstract}
Several primal and dual quantitative characterizations of regularity properties of collections of sets in normed linear spaces are discussed. Relationships between regularity properties of collections of sets and those of set-valued mappings are provided.
\end{abstract}


\section{Introduction}

Regularity properties of collections of sets play an important role in variational analysis and optimization, particularly as constraint qualifications in establishing optimality conditions and coderivative/subdifferential calculus and in analyzing convergence of numerical algorithms.

The concept of linear regularity was introduced in \cite{BauBor93,BauBor96} as a key condition in establishing linear convergence rates of sequences generated by the cyclic projection algorithm for finding the projection of a point on the intersection of a collection of closed convex sets.
This property has proved to be an important qualification condition in the convergence analysis, optimality conditions, and subdifferential calculus; cf., e.g.,
\cite{
BakDeuLi05,
BauBorLi99, BauBorTse00,BurDeng05,Iof89,
LiNgPon07,NgaiThe01,
ZheNg08,ZheWeiYao10}.

Recently, when investigating the extremality, stationarity and regularity properties of collections of sets systematically, several other kinds of regularity properties have been considered in \cite{Kru05.1,Kru06.1,Kru09.1, KruLop12.1,KruLop12.2,KruTha13.1}.
They have proved to be useful in convergence analysis \cite{AttBolRedSou10,KruTha13.1,LewLukMal09,Luk12,Luk13, HesLuk,HesLuk13} and are closely related to certain stationarity properties involved in extensions of the extremal principle \cite{Mor06.1,Kru03.1,Kru04.1,Kru09.1, KruLop12.1}.

In this study, we aim at providing primal and dual quantitative characterizations of several regularity properties of collections of sets.
We also discuss their relationships with the corresponding regularity properties of set-valued mappings.

After introducing in the next section some basic notation, we discuss in Section~\ref{URtheory} three primal space local regularity properties of collections of sets, namely, \emph{semiregularity}, \emph{subregularity}, and \emph{uniform regularity} as well as their quantitative characterizations.
The main result of this section -- Theorem~\ref{T2.1} -- gives equivalent metric characterizations of the three mentioned regularity properties.
Section~\ref{S4} is dedicated to dual characterizations of the regularity properties.
In Theorem~\ref{T3.1}~(i), we give a sufficient condition of subregularity in terms of Fr\'echet normals.
In Section~\ref{S5}, we present relationships between regularity properties of collections of sets and the corresponding regularity properties of set-valued mappings.

\section{Notation}

Our basic notation is standard; cf. \cite{Mor06.1,RocWet98}.
For a normed linear space $X$, its topological dual is denoted $X^*$, while $\langle\cdot,\cdot\rangle$ denotes the bilinear form defining the pairing between the two spaces.
The closed unit ball in a normed space is denoted $\B $, $B_\delta(x)$ stands for the closed ball with radius $\delta$ and centre $x$.
Products of normed spaces will be considered with the maximum type norms, if not specified otherwise.

The Fr\'echet normal cone to a set $\Omega\subset X$ at $x\in \Omega$ and the Fr\'echet subdifferential of a function $f:X \to\R_\infty:=\R\cup\{+\infty\}$ at a point $x$ with $f(x)<\infty$ are defined, respectively, by
\begin{equation*}
N_{\Omega}(x):= \left\{x^* \in X^* :\; \limsup_{u\to{x},\,u\in\Omega\setminus\{x\}} \frac {\langle x^*,u-x \rangle}{\|u-x\|} \leq 0 \right\},
\end{equation*}
\begin{equation*}
\partial f(x):=\left\{x^*\in X^*:\; \liminf_{\substack{u\to x,\,u\neq x}} \frac{f(u)-f(x)-\langle x^*,u-x\rangle}{\norm{u-x}}\ge 0\right\}.
\end{equation*}

For a given set $\Omega\subset X$, the distance function associated with $\Omega$ is defined by
$$
d(x,\Omega):=\inf_{\omega\in\Omega}\norm{x-\omega},\; \forall x\in X.
$$

In the sequel, $\bold{\Omega}$ stands for a collection of $m$ $(m\ge 2)$ sets $\Omega_1,\ldots,\Omega_m$ in a normed linear space $X$, and we assume the existence of a point $\bar x\in \bigcap_{i=1}^m\Omega_i$.

\section{Regularity Properties of Collections of Sets}\label{URtheory}

In this section, we discuss local primal space regularity properties of finite collections of sets and their primal space characterizations.

\subsection{Definitions}

The next definition introduces several regularity properties of $\bold{\Omega}$ at $\bar x$.

\begin{dfn}\label{HolderUR}
\begin{enumerate}
\item
$\bold{\Omega}$ is semiregular at $\bar x$ iff there exist positive numbers $\alpha$ and $\delta$ such that
\begin{equation}\label{HR}
\bigcap_{i=1}^m(\Omega_i-x_i)\  \bigcap\ B_{\rho}(\bar x) \neq \emptyset
\end{equation}
for all $\rho \in ]0,\delta[$ and all
$x_i\in{X}$ $(i=1,\ldots,m)$ such that $\max\limits_{1\leq i\leq m}\|x_i\| \le \alpha\rho$.
\item
$\bold{\Omega}$ is subregular at $\bar x$ iff there exist positive numbers $\alpha$ and $\delta$ such that
\begin{equation}\label{Hlr'}
\bigcap_{i=1}^m\left(\Omega_i+ (\alpha\rho)\B \right)\bigcap B_{\delta}(\bar x)\subseteq
\left(\bigcap_{i=1}^m \Omega_i\right)+\rho\B
\end{equation}
for all $\rho\in ]0,\delta[$.
\item
$\bold{\Omega}$ is uniformly regular at $\bar x$ iff there exist positive numbers $\alpha$ and $\delta$ such that
\begin{equation}\label{HUR}
\bigcap_{i=1}^m(\Omega_i-\omega_i-x_i) \bigcap(\rho\B )\neq \emptyset
\end{equation}
for all $\rho \in ]0,\delta[$, $\omega_i \in \Omega_i \cap B_{\delta}(\bar{x})$, and all $x_i\in{X}$ $(i=1,\ldots,m)$ such that $\max\limits_{1\leq i\leq m}\|x_i\| \le \alpha\rho$.
\end{enumerate}
\end{dfn}

\begin{rem}\label{rem1}
Among the three regularity properties in Definition~\ref{HolderUR}, the third one is the strongest.
Indeed, condition \eqref{HR} corresponds to taking $\omega_i=\bar x$ in \eqref{HUR}.
To compare properties (ii) and (iii), it is sufficient to notice that condition \eqref{Hlr'} is equivalent to the following one: for any $x\in B_{\delta}(\bar x)$, $\omega_i \in \Omega_i$,  $x_i\in{X}$ $(i=1,\ldots,m)$ such that $\max\limits_{1\leq i\leq m}\|x_i\| \le \alpha\rho$, and $\omega_i+x_i=x$ $(i=1,\ldots,m)$, it holds
\begin{equation*}
\bigcap_{i=1}^m(\Omega_i-x)\bigcap (\rho\B )\neq \emptyset.
\end{equation*}
This corresponds to taking $\omega_i+x_i=x$ $(i=1,\ldots,m)$ in \eqref{HUR} (with $x\in X$) and possibly choosing a smaller $\delta>0$.
Hence, (iii) $\Longrightarrow$ (i) and (iii) $\Longrightarrow$ (ii).
\end{rem}

\begin{rem}
When $\bar x\in \ptrog \bigcap_{i=1}^m\Omega_i$, all the properties in Definition~\ref{HolderUR} hold true automatically.
\end{rem}

\begin{rem}\label{rem3}
\hspace{-.1cm}\footnote{Observed by a reviewer.}
When $\Omega_1=\Omega_2=\ldots=\Omega_m$, property (ii) in Definition~\ref{HolderUR} is trivially satisfied (with $\alpha=1$).
\end{rem}

The regularity properties in Definition~\ref{HolderUR} can be equivalently defined using the following nonnegative constants which provide quantitative characterizations of these properties:
\begin{gather}\label{01.1}
\theta[\bold{\Omega}](\bar{x}):= \liminf_{\rho \downarrow 0} \dfrac{\theta_{\rho}[\bold{\Omega}](\bar{x})}{\rho},
\\\label{01.2}
\zeta[\bold{\Omega}](\bar{x}):= \lim_{\delta\downarrow 0}\inf_{0<\rho<\delta} \dfrac{\zeta_{\rho,\delta}[\bold{\Omega}](\bar{x})} {\rho},
\\\label{01.3}
\hat{\theta}[\bold{\Omega}](\bar{x}):= \liminf_{\omega_i\TO{\Omega_i}\bar{x},\rho\downarrow0} \dfrac{\theta_{\rho}[\Omega_1-\omega_1,\ldots,\Omega_m-\omega_m](0)}{\rho},
\end{gather}
where, for $\rho>0$ and $\delta>0$,
\begin{gather}\label{02.1}
\theta_{\rho}[\bold{\Omega}](\bar{x}):= \sup\left\{r \ge 0 :\; \bigcap_{i=1}^m (\Omega_i - x_i) \bigcap B_{\rho}(\bar{x}) \neq \emptyset,\; \forall x_i \in r\B \right\},
\\\label{02.2}
\zeta_{\rho,\delta}[\bold{\Omega}](\bar{x}):= \sup\left\{r\ge0 :\;\bigcap_{i=1}^m (\Omega_i+ r\B )\bigcap B_{\delta}(\bar x)\subseteq
\bigcap_{i=1}^m \Omega_i+\rho\B \right\}.
\end{gather}

The next proposition follows immediately from the definitions.

\begin{prop}\label{theorem11}
\begin{enumerate}
\item
$\bold{\Omega}$ is semiregular at $\bar x$ if and only if $\theta[\bold{\Omega}](\bar{x})>0$. Moreover, $\theta[\bold{\Omega}](\bar{x})$ is the exact upper bound of all numbers $\alpha$ such that (\ref{HR}) is satisfied.
\item
$\bold{\Omega}$ is subregular at $\bar x$ if and only if $\zeta[\bold{\Omega}](\bar{x})>0$. Moreover, $\zeta[\bold{\Omega}](\bar{x})$ is the exact upper bound of all numbers $\alpha$ such that (\ref{Hlr'}) is satisfied.
\item
$\bold{\Omega}$ is uniformly regular at $\bar x$ if and only if $\hat{\theta}[\bold{\Omega}](\bar{x})>0$. Moreover, $\hat{\theta}[\bold{\Omega}](\bar{x})$ is the exact upper bound of all numbers $\alpha$ such that (\ref{HUR}) is satisfied.
\end{enumerate}
\end{prop}

\begin{rem}
Properties (i) and (iii) in Definition~\ref{HolderUR} were discussed in \cite{Kru06.1} (where they were called \emph{regularity} and \emph{strong regularity}, respectively) and \cite{Kru09.1} (properties (R)$_S$ and (UR)$_S$) and \cite{KruLop12.1} (\emph{regularity} and \emph{uniform regularity}).
The current terminology used in parts (i) and (ii) of Definition~\ref{HolderUR} comes from the standard terminology used for the corresponding regularity properties of set-valued mappings; cf. Section~\ref{S5}.

Constants \eqref{01.1}, \eqref{01.3}, and \eqref{02.1} can be traced back to \cite{Kru96.2,Kru98.1,Kru00.1,Kru02.1,Kru03.1,Kru04.1,Kru05.1}.
Property (ii) in Definition~\ref{HolderUR} and constants \eqref{01.2} and \eqref{02.2} are new.
\end{rem}

\begin{rem}\label{newrem}
If finite, constants $\zeta[\bold{\Omega}](\bar{x})$ and $\hat{\theta}[\bold{\Omega}](\bar{x})$ always take values in $[0,1]$, while constant $\theta[\bold{\Omega}](\bar{x})$ can be strictly greater than one (cf. Example \ref{newexa} below).
In view of Remark \ref{rem1}, it is not difficult to check that
$\hat{\theta}[\bold{\Omega}](\bar{x})\le \min\{\theta[\bold{\Omega}](\bar{x}), \zeta[\bold{\Omega}](\bar{x})\}$.
\end{rem}

The equivalent representation of constant \eqref{02.1} given in the next proposition can be useful.

\begin{prop}\label{pr1}
For any $\rho>0$,
\begin{gather}\label{02.1+}
\theta_{\rho}[\bold{\Omega}](\bar{x})= \sup\left\{r \ge 0 :\; r\B ^m\subseteq\bigcup_{x\in B_{\rho}(\bar{x})}\prod_{i=1}^m (\Omega_i-x)\right\},
\end{gather}
where $\prod_{i=1}^m(\Omega_i-x) =(\Omega_1-x)\times\ldots\times(\Omega_m-x)$ and $\B ^m=\prod_{i=1}^m\B $.
\end{prop}
\begin{proof}
It is sufficient to observe that condition
$$\bigcap_{i=1}^m (\Omega_i - x_i) \bigcap B_{\rho}(\bar{x}) \neq \emptyset$$
in \eqref{02.1} is equivalent to the existence of $x\in B_{\rho}(\bar{x})$ such that $x_i\in\Omega_i - x$ for all $i=1,\ldots,m$.
This holds true for all $x_i \in r\B$ if and only if
\begin{gather*}
r\B ^m\subseteq\bigcup_{x\in B_{\rho}(\bar{x})}\prod_{i=1}^m (\Omega_i-x).
\end{gather*}
\qed\end{proof}

From Propositions~\ref{theorem11} and \ref{pr1}, we immediately obtain equivalent representations of semiregularity and uniform regularity.
\begin{cor}
\begin{enumerate}
\item
$\bold{\Omega}$ is semiregular at $\bar x$ if and only if there exist  positive numbers $\alpha$ and $\delta$ such that
\begin{equation}\label{HR2}
(\alpha\rho) \mathds{B}^m\subseteq \bigcup_{x\in B_{\rho}(\bar{x})}\prod_{i=1}^m (\Omega_i-x)
\end{equation}
for all $\rho\in]0,\delta[$.
Moreover, $\theta[\bold{\Omega}](\bar{x})$ is the exact upper bound of all numbers $\alpha$ such that (\ref{HR2}) is satisfied.
\item
$\bold{\Omega}$ is uniformly regular at $\bar x$ if and only if there exist  positive numbers $\alpha$ and $\delta$ such that
\begin{equation}\label{umi2}
(\alpha\rho) \mathds{B}^m\subseteq \bigcap_{\substack{\omega_i\in\Omega_i\cap B_{\delta}(\bar x)\\ (i=1,\ldots,m)}} \bigcup_{x\in\rho\B }\prod_{i=1}^m (\Omega_i-\omega_i-x)
\end{equation}
for all $\rho\in]0,\delta[$.
Moreover, $\hat\theta[\bold{\Omega}](\bar{x})$ is the exact upper bound of all numbers $\alpha$ such that (\ref{umi2}) is satisfied.
\end{enumerate}
\end{cor}

\begin{rem}
The definition of subregularity in Definition~\ref{HolderUR} (ii) is already of inclusion type in the setting of the original space $X$.
There is no need to consider the product space $X^m$.
\end{rem}

\subsection{Examples}\label{examples}

We next present examples illustrating that properties (i) and (ii) in Definition~\ref{HolderUR} are in general independent and none of these two properties implies property (iii) in Definition~\ref{HolderUR}.

\begin{examp}\label{exa1}
In the real plane ${\R}^2$ with the Euclidean norm, consider two sets
$$
\Omega_1=\Omega_2:= \R\times\{0\}
$$
and the point $\bar x = (0,0) \in \Omega_1 \cap \Omega_2$.
The collection $\{\Omega_1,\Omega_2\}$ is subregular at $\bar x$, while it is not semiregular at this point.
\end{examp}
\begin{proof}
In view of Remark~\ref{rem3}, $\{\Omega_1,\Omega_2\}$ is subregular at $\bar x$.
Observe also that ${(\Omega_1-(0,-\varepsilon))} \cap (\Omega_2-(0,\varepsilon))=\emptyset$ for any $\varepsilon>0$. Hence, by \eqref{02.1} and \eqref{01.1}, $\{\Omega_1,\Omega_2\}$ is not semiregular at $\bar x$.
\sloppy
\qed
\end{proof}
\begin{examp}\label{exa2}
In the real plane ${\R}^2$ with the Euclidean norm, consider two sets
$$
\Omega_1:= \left\{(u,v) \in {\R}^2 :\; u\le0 \mbox{ or } v\ge u^2\right\},
\;
\Omega_2:= \left\{(u,v) \in {\R}^2 :\; u\le0 \mbox{ or } v\le 0\right\}
$$
and the point $\bar x = (0,0) \in \Omega_1 \cap \Omega_2$.
The collection $\{\Omega_1,\Omega_2\}$ is semiregular at $\bar x$, while it is not subregular at this point.
\end{examp}
\begin{proof}
We first show that $\{\Omega_1,\Omega_2\}$ is semiregular at $\bar x$.
For any number $\rho>0$, we set $x_{\rho}:=(-\rho,0)$.
Then
$B_{\rho}(x_{\rho})\subseteq \Omega_i$, i.e., $x_{\rho}+x_i\in\Omega_i$ for any $x_i\in\rho\B$ $(i=1,2)$, and consequently
$$
x_{\rho}\in(\Omega_1-x_1)\cap (\Omega_2-x_2)\cap B_{\rho}(\bar{x}),\; \forall x_i\in\rho\B\;(i=1,2).
$$
Hence, $\theta_{\rho}[\{\Omega_1,\Omega_2\}](\bar{x})\ge\rho$ and $\theta[\{\Omega_1,\Omega_2\}](\bar{x})\ge1$. (One can show that these are actually equalities.)
Thus, $\{\Omega_1,\Omega_2\}$ is semiregular at $\bar x$.

Suppose that inclusion \eqref{Hlr'} holds for some positive numbers $\alpha$ and $\delta$ and all $\rho\in ]0,\delta[$.
Set $\rho_n:=\frac{1}{n}$ and $x_n:=\left(\sqrt{\alpha\rho_n},\alpha\rho_n\right)$.
Then $x_n\in\left(\Omega_1+ (\alpha\rho_n)\B \right) \bigcap\left(\Omega_2+ (\alpha\rho_n)\B \right)$, $d(x_n,\Omega_1\bigcap\Omega_2)=\sqrt{\alpha\rho_n}$ and, for sufficiently large $n$, $\rho_n<\delta$ and $x_n\in B_{\delta}(\bar x)$.
It follows from \eqref{Hlr'} that $\sqrt{\alpha\rho_n}\le\rho_n$, and consequently $\alpha\le\rho_n$.
This yields $\alpha\le0$ which contradicts the assumptions.
Hence, $\{\Omega_1,\Omega_2\}$ is not subregular at $\bar x$.
\qed
\end{proof}
\begin{examp}
In the real plane ${\R}^2$ with the Euclidean norm, consider two sets
$$
\Omega_1=\Omega_2:= \left\{(u,v) \in {\R}^2 :\; u\le0 \mbox{ or } v=0\right\}
$$
and the point $\bar x = (0,0) \in \Omega_1 \cap \Omega_2$.
The collection $\{\Omega_1,\Omega_2\}$ is both semiregular and subregular at $\bar x$, while it is not uniformly regular at this point.
\end{examp}
\begin{proof}
In view of Remark~\ref{rem3}, $\{\Omega_1,\Omega_2\}$ is subregular at $\bar x$.
Using the arguments from the first part of Example \ref{exa2}, it is easy to check that the collection is semiregular at $\bar x$.
We next show that $\{\Omega_1,\Omega_2\}$ is not uniformly regular at this point. Indeed, for any given numbers $\delta,\alpha>0$, we find positive numbers $\rho<r<\delta$ and take $$\omega_i=(r,0){\in\Omega_i}\cap B_{\delta}(\bar x)\; (i=1,2),\quad a_1=(0,\alpha\rho),\quad a_2=(0,-\alpha\rho)\in \alpha\rho\B.$$
We have
$$
(\Omega_1-\omega_1-a_1)\bigcap (\Omega_2-\omega_2-a_2)\bigcap (\rho\B )=\{(u,v)\in {\R}^2:\; u\le -r\}\bigcap (\rho\B )=\emptyset.
$$
\qed
\end{proof}

The following example demonstrates that the constant $\theta[\bold{\Omega}](\bar{x})$ can take values greater than one.
\begin{examp}\label{newexa}
In the real plane ${\R}^2$ with the Euclidean norm, consider two sets
$$
\Omega_1:={\R}^2,\Omega_2:= \left\{(u,v) \in {\R}^2 :\; u-\sqrt{3}v\ge 0 \mbox{ or } u+\sqrt{3}v\ge 0\right\}
$$
and the point $\bar x = (0,0) \in \Omega_1 \cap \Omega_2$.
Then, $\theta[\bold{\Omega}](\bar{x})=2>1$.
\end{examp}
\begin{proof}
By the structure of the sets, we have
\begin{align*}
\theta_{\rho}[\bold{\Omega}](\bar{x})
&=\sup\{r\ge 0:\; (\Omega_2-x) \bigcap (\rho\B )\neq \emptyset,\; \forall x\in r\B \}\\
&=\sup\{r\ge 0:\; d(0,\Omega_2-x)\le\rho,\; \forall x\in r\B \}\\
&=\sup\{r\ge 0:\; d(x,\Omega_2)\le\rho,\; \forall x\in r\B \}\\
&=\sup\{r\ge 0:\; \max\{d(x,\Omega_2):\; x\in r\B \}\le\rho\}\\
&=\sup\{r\ge 0:\; \frac{r}{2}\le\rho\}=2\rho.
\end{align*}
The second last equality holds true since for any $r>0$,
$$
\max\{d(x,\Omega_2):\; x\in r\B \}=d(x_r,\Omega_2)=\frac{r}{2},
$$
where $x_r:=(-r,0)$.

Hence, by definition,
$$
\theta[\bold{\Omega}](\bar{x})= \liminf_{\rho \downarrow 0} \dfrac{\theta_{\rho}[\bold{\Omega}](\bar{x})}{\rho}=2.
$$
\qed
\end{proof}

\subsection{Metric characterizations}\label{metric}

The regularity properties of collections of sets  in Definition~\ref{HolderUR} can also be characterized in metric terms.
The next proposition provides equivalent metric representations of constants \eqref{01.1} -- \eqref{01.3}.

\begin{prop}\label{pr2}
\begin{align}\label{newcon}
\theta[\bold{\Omega}](\bar{x})&= \liminf_{\substack{x_i\to 0\; (1\le i\le m)\\ \bar x\notin \bigcap_{i=1}^m(\Omega_i-x_i)}}
\frac{\max_{1\le i\le m}\norm{x_i}}{d\Big{(}\bar x,\bigcap_{i=1}^m(\Omega_i-x_i)\Big{)}},
\\\label{errcon'}
\zeta[\bold{\Omega}](\bar{x})&= \quad\liminf_{\substack{x\to\bar{x}\\ x\notin \bigcap_{i=1}^m\Omega_i}}
\quad\frac{\max_{1\le i\le m} d(x,\Omega_i)}{d\Big{(}x,\bigcap_{i=1}^m\Omega_i\Big{)}}
\\\notag
&=\liminf_{\substack{x\to\bar{x}\\ \omega_i\TO{\Omega_i}\bar x\; (1\le i\le m)\\ x\notin \bigcap_{i=1}^m\Omega_i}}\frac{\max_{1\le i\le m}\norm{\omega_i-x}}{d\Big{(}x,\bigcap_{i=1}^m\Omega_i\Big{)}},
\\\label{vartheta}
\hat{\theta}[\bold{\Omega}](\bar{x})&= \liminf_{\substack{x\to\bar{x}\\x_i \to 0\;(1\le i\le m)\\x\notin\bigcap_{i=1}^m(\Omega_i-x_i)}} \frac{\max_{1\le i\le m}d(x+x_i,\Omega_i)} {d\Big{(}x,\bigcap_{i=1}^m(\Omega_i-x_i)\Big{)}}
\\\notag
&= \liminf_{\substack{x\to\bar{x}\\x_i\to0,\; \omega_i\TO{\Omega_i} \bar x\;(1\le i\le m)\\ x\notin\bigcap_{i=1}^m(\Omega_i-x_i)}} \frac{\max_{1\le i\le m} \norm{x+x_i-\omega_i}} {d\Big{(}x,\bigcap_{i=1}^m(\Omega_i-x_i)\Big{)}}.
\end{align}
\end{prop}

\begin{proof}
\emph{Equality \eqref{newcon}}.
Let $\xi$ stand for the \RHS\ of \eqref{newcon}.
Suppose that $\xi>0$ and fix an arbitrary number $\gamma\in ]0,\xi[$.
Then there is a number $\delta>0$ such that
\begin{equation}\label{new1}
\gamma d\left(\bar x,\bigcap_{i=1}^m(\Omega_i-x_i)\right)\le \max_{1\le i\le m}\norm{x_i},\; \forall x_i\in \delta\B \; (i=1,\ldots,m).
\end{equation}
Choose a number $\alpha\in ]0,\gamma[$ and set $\delta'=\frac{\delta}{\alpha}$.
Then, for any $\rho\in ]0,\delta'[$ and $x_i\in (\alpha\rho)\B$ $(i=1,\ldots,m)$, it holds
$\max_{1\le i\le m}\norm{x_i}\le \alpha\rho\le \alpha\delta'=\delta$.
Hence, \eqref{new1} yields
$$
d\left(\bar x,\bigcap_{i=1}^m(\Omega_i-x_i)\right)\le \frac{1}{\gamma}\max_{1\le i\le m}\norm{x_i}\le \frac{\alpha}{\gamma}\rho<\rho.
$$
This implies \eqref{HR} and consequently $\theta[\bold{\Omega}](\bar{x})\ge\alpha$.
Taking into account that $\alpha$ can be arbitrarily close to $\xi$, we obtain $\theta[\bold{\Omega}](\bar{x})\ge \xi$.

Conversely, suppose that $\theta[\bold{\Omega}](\bar{x})>0$ and fix an arbitrary number $\alpha\in ]0,\theta[\bold{\Omega}](\bar{x})[$.
Then there is a number $\delta>0$ such that \eqref{HR} is satisfied for all $\rho\in ]0,\delta[$ and $x_i\in (\alpha\rho)\B$ $(i=1,\ldots,m)$.
Choose a positive $\delta'<\alpha\delta$.
For any $x_i\in \delta'\B \; (i=1,\ldots,m)$, it holds $\max_{1\le i\le m}\norm{x_i}<\alpha\delta$.
Pick up a $\rho\in]0,\delta[$ such that $\max_{1\le i\le m}\norm{x_i}=\alpha\rho$.
Then \eqref{HR} yields
$$
\alpha d\left(\bar x,\bigcap_{i=1}^m(\Omega_i-x_i)\right)\le \alpha\rho=\max_{1\le i\le m}\norm{x_i}.
$$
This implies $\xi\ge \alpha$. Since $\alpha$ can be arbitrarily close to $\theta[\bold{\Omega}](\bar{x})$, we deduce
$\xi\ge \theta[\bold{\Omega}](\bar{x})$.

\emph{Equality \eqref{errcon'}}.
Let $\xi$ stand for the \RHS\ of \eqref{errcon'}.
Suppose that $\xi>0$ and fix an arbitrary number $\alpha\in ]0,\xi[$. Then there is a number $\delta>0$ such that
\begin{equation*}
\alpha d\left(x,\bigcap_{i=1}^m\Omega_i\right) \le \max_{1\le i\le m} d(x,\Omega_i),\; \forall x \in B_{\delta}(\bar{x}).
\end{equation*}
If $x\in\bigcap_{i=1}^m\left(\Omega_i+ (\alpha\rho)\B \right)\bigcap B_{\delta}(\bar x)$ for some $\rho\in ]0,\delta[$, then $\max_{1\le i\le m} d(x,\Omega_i)\le\alpha\rho$, and consequently $d\left(x,\bigcap_{i=1}^m\Omega_i\right) \le\rho$, i.e., $\zeta_{\rho,\delta}[\bold{\Omega}](\bar{x}) \ge \alpha\rho$.
Hence, $\zeta[\bold{\Omega}](\bar{x})\ge \alpha$. Since $\alpha$ can be arbitrarily close to $\xi$, we obtain $\zeta[\bold{\Omega}](\bar{x})\ge \xi$.

Conversely, suppose that $\zeta[\bold{\Omega}](\bar{x})>0$ and fix any $\alpha\in ]0,\zeta[\bold{\Omega}](\bar{x})[$.
Then there is a number $\delta>0$ such that
\eqref{Hlr'} is satisfied for all $\rho\in ]0,\delta[$.
Choose a positive number $\delta'<\min\{\alpha\delta,\delta\}$.
For any $x \in B_{\delta'}(\bar{x})$, it holds
$$
\max_{1\le i\le m}d(x,\Omega_i)\le \norm{x-\bar x}\le \delta'<\alpha\delta.
$$
Choose a $\rho\in]0,\delta[$ such that $\max_{1\le i\le m}d(x,\Omega_i)=\alpha\rho$.
Then, by \eqref{Hlr'},
$$
\alpha d\left(x,\bigcap_{i=1}^m\Omega_i\right) \le \alpha\rho =\max_{1\le i\le m}d(x,\Omega_i).
$$
Hence, $\alpha\le\xi$.
By letting $\alpha\to\zeta[\bold{\Omega}](\bar{x})$, we obtain $\zeta[\bold{\Omega}](\bar{x})\le\xi$.

\emph{Equality \eqref{vartheta}} has been proved in \cite[Theorem 1]{Kru05.1}.
\qed
\end{proof}

Propositions~\ref{theorem11} and \ref{pr2} imply equivalent metric characterizations of the regularity properties of collections of sets.

\begin{thm}\label{T2.1}
\begin{enumerate}
\item
$\bold{\Omega}$ is semiregular at $\bar x$ if and only if there exist  positive numbers $\gamma$ and $\delta$ such that
\begin{equation}\label{Hri}
    \gamma d\left(\bar x,\bigcap_{i=1}^m(\Omega_i-x_i)\right) \le \max_{1\le i\le m}\norm{x_i},\; \forall x_i \in \delta\B \; (i=1,\ldots,m).
\end{equation}
Moreover, $\theta[\bold{\Omega}](\bar{x})$ is the exact upper bound of all numbers $\gamma$ such that (\ref{Hri}) is satisfied.
\item
$\bold{\Omega}$ is subregular at $\bar x$ if and only if there exist  positive numbers $\gamma$ and $\delta$ such that
\begin{equation}\label{Hlr}
    \gamma d\left(x,\bigcap_{i=1}^m\Omega_i\right) \le \max_{1\le i\le m} d(x,\Omega_i),\; \forall x \in B_{\delta}(\bar{x}).
\end{equation}
Moreover, $\zeta[\bold{\Omega}](\bar{x})$ is the exact upper bound of all numbers $\gamma$ such that (\ref{Hlr}) is satisfied.
\item
$\bold{\Omega}$ is uniformly regular at $\bar x$ if and only if there exist  positive numbers $\gamma$ and $\delta$ such that
\begin{equation}\label{umi}
    \gamma d\left(x,\bigcap_{i=1}^m(\Omega_i-x_i)\right) \le \max_{1\le i\le m}d(x+x_i,\Omega_i)
\end{equation}
for any $x \in B_{\delta}(\bar{x})$, $x_i \in \delta\B \; (i=1,\ldots,m)$.
Moreover, $\hat\theta[\bold{\Omega}](\bar{x})$ is the exact upper bound of all numbers $\gamma$ such that (\ref{umi}) is satisfied.
\end{enumerate}
\end{thm}

\begin{rem}
Property (\ref{Hlr}) in the above theorem (also known as local \emph{linear regularity}, \emph{linear coherence}, or \emph{metric inequality}) has been around for more than 20 years; cf. \cite{AusDanThi05,BakDeuLi05,BauBor93,BauBor96,BauBorLi99, BauBorTse00,BurDeng05,Iof89,Iof00_,KlaLi99,LewPan98,
LiNahSin00,LiNgPon07,NgaiThe01,Pen13,ZheNg08,ZheWeiYao10}.
It has been used as a key condition when establishing linear convergence rates of sequences generated by cyclic projection algorithms and a qualification condition for subdifferential and normal cone calculus formulae.
The stronger property (\ref{umi}) is sometimes referred to as \emph{uniform metric inequality} \cite{Kru05.1,Kru06.1,Kru09.1}.
Property (\ref{Hri}) seems to be new.
\end{rem}

\section{Dual Characterizations}\label{S4}

This section discusses dual characterizations of regularity properties of a collection of sets $\bold{\Omega}:=\{\Omega_1,\ldots,\Omega_m\}$ at $\bar x\in \bigcap_{i=1}^m\Omega_i$.
We are going to use the notation ${\widehat\Omega:= \Omega_1\times\ldots\times\Omega_m\subset X^m}$.
\sloppy

Recall that the (normalized) \emph{duality mapping} \cite[Definition~3.2.6]{Lucc06} $J$ between a normed space $Y$ and its dual $Y^*$ is defined as
$$
J(y):=\left\{y^*\in \mathbb{S}_{Y^*}:\; \langle y^*,y\rangle=\norm{y}\right\},\; \forall y\in Y.
$$
Note that $J(-y)=-J(y)$.

The following simple fact of convex analysis is well known (cf., e.g., \cite[Corollary~2.4.16]{Zal02}).

\begin{lem}\label{ll01}
Let $(Y,\|\cdot\|)$ be a normed space.
\begin{enumerate}
\item
$\sd\|\cdot\|(y)=J(y)$ for any $y\ne0$.
\item
$\sd\|\cdot\|(0)=\B^*$.
\end{enumerate}
\end{lem}

Making use of the convention that the topology in $X^m$ is defined by the maximum type norm,
it is not difficult to establish a representation of the duality mapping on $X^m$.
\begin{prop}\label{prJ}
For each $(x_1,\ldots,x_m)\in X^m$,
\begin{multline*}
J(x_1,\ldots,x_m)= \Biggl\{(x_1^*,\ldots,x_m^*)\in(X^*)^m:\; \sum_{i=1}^m\|x_i^*\|=1;\mbox{ either } x_i^*=0\\
\mbox{or }\left(\|x_i\|=\max_{1\le j\le m} \|x_j\|,\; x_i^*\in\|x_i^*\| J(x_i)\right)\; (i=1,\ldots,m)\Biggr\}.
\end{multline*}
\end{prop}
\begin{proof}
Let $\hat x:=(x_1,\ldots,x_m)\in X^m$, $\hat x^*:=(x_1^*,\ldots,x_m^*)\in(X^*)^m$.
Then
$${\|\hat x\|=\max_{1\le i\le m} \|x_i\|},\quad \|\hat x^*\|=\sum_{i=1}^m\|x_i^*\|,\quad \langle \hat x^*,\hat x\rangle=\sum_{i=1}^m\langle x_i^*,x_i\rangle.$$
Suppose $\|\hat x^*\|=1$, i.e., $\sum_{i=1}^m\|x_i^*\|=1$.
Then $\hat x^*\in J(\hat x)$ if and only if ${\sum_{i=1}^m\langle x_i^*,x_i\rangle=\|\hat x\|}$.
In its turn, the last equality holds true if and only if ${\langle x_i^*,x_i\rangle=\|x_i^*\|\cdot\|\hat x\|}$ for all $i=1,\ldots,m$.

Indeed, if $\langle x_i^*,x_i\rangle=\|x_i^*\|\cdot\|\hat x\|$ for all $i=1,\ldots,m$, then adding these $m$ equalities, we obtain $\sum_{i=1}^m\langle x_i^*,x_i\rangle=\|\hat x\|$.
Conversely, if $\langle x_i^*,x_i\rangle\ne\|x_i^*\|\cdot\|\hat x\|$, i.e., ${\langle x_i^*,x_i\rangle<\|x_i^*\|\cdot\|\hat x\|}$ for some $i\in\{1,\ldots,m\}$, then
$$
\sum_{j=1}^m\langle x_j^*,x_j\rangle =\langle x_i^*,x_i\rangle+\sum_{j\ne i}\langle x_j^*,x_j\rangle <\|x_i^*\|\cdot\|\hat x\|+\|\hat x\|\sum_{j\ne i} \|x_j^*\|=\|\hat x\|.
$$

Finally, $\langle x_i^*,x_i\rangle=\|x_i^*\|\cdot\|\hat x\|$ for some $i\in\{1,\ldots,m\}$ if and only if either $\|x_i\|=\|\hat x\|$ and $x_i^*\in\|x_i^*\| J(x_i)$ or $x_i^*=0$.
\qed\end{proof}

In this section, along with the maximum type norm on $X^{m+1}=X\times X^m$, we are going to use another one depending on a parameter $\rho>0$ and defined as follows:
\begin{equation}\label{norm}
\norm{(x,\hat x)}_{\rho} :=\max\left\{\norm{x},\rho \norm{\hat x}\right\},\quad x\in X,\;\hat x\in X^m.
\end{equation}
It is easy to check that the corresponding dual norm has the following representation:
\begin{equation}\label{normd}
\norm{(x^*,\hat x^*)}_{\rho} =\|x^*\|+\rho^{-1}\|\hat x^*\|,\quad x^*\in X^*,\;\hat x^*\in (X^m)^*.
\end{equation}
Note that if, in \eqref{norm} and \eqref{normd}, $\hat x= (x_1,\ldots,x_m)$ and $\hat x^*= (x_1^*,\ldots,x_m^*)$ with $x_i\in X$ and $x_i^*\in X^*$ ($i=1,2,\ldots,m$), then $\|\hat x\|=\max_{1\le i\le m}\|x_i\|$ and $\|\hat x^*\|=\sum_{i=1}^m\|x_i^*\|$.

The next few facts of subdifferential calculus are used in the proof of the main theorem below.

\begin{lem}\label{ll02}
Let $X$ be a normed space and
$\varphi(u,\hat u)=\|(u-u_1,\ldots,u-u_m)\|$, $u\in X$, $\hat u:=(u_1,\ldots,u_m)\in X^m$.
Suppose $x\in X$, $\hat x:=(x_1,\ldots,x_m)\in X^m$, and ${\hat v:=(x-x_1,\ldots,x-x_m)}\ne0$. Then
\begin{align*}
\sd\varphi(x,\hat x)\subseteq\bigl\{\left(x^*,\hat x^* =(x_1^*,\ldots,x_m^*)\right)\in &X^*\times (X^*)^m :\;\\
&-\hat x^*\in J(\hat v),\; x^*=-(x_{1}^*+\ldots+x_{m}^*)\bigr\}.
\end{align*}
\end{lem}
\begin{proof}
Let $\left(x^*,\hat x^* =(x_1^*,\ldots,x_m^*)\right)\in \sd\varphi(x,\hat x)$, i.e.,
$$
\|(u-u_1,\ldots,u-u_m)\|-\|(x-x_1,\ldots,x-x_m)\|\ge \langle x^*,u-x\rangle+\sum_{i=1}^m\langle x_i^*,u_i-x_i\rangle
$$
for any $u\in X$ and $\hat u:=(u_1,\ldots,u_m)\in X^m$.
In particular, with $u=x$ and $u_i=x_i-x_i'$ ($i=1,\ldots,m$) for an arbitrary $\hat x':=(x_1',\ldots,x_m')\in X^m$, we have
$$
\|\hat v+\hat x'\|-\|\hat v\|\ge-\langle\hat x^*,\hat x'\rangle,
$$
i.e., $-\hat x^*\in J(\hat v)$.
Similarly, with $u=x+x'$ and $u_i=x_i+x'$ ($i=1,\ldots,m$) for an arbitrary $x'\in X$, we have
$$
\left\langle x^*+\sum_{i=1}^mx_i^*,x'\right\rangle\le0,
$$
and consequently $x^*+x_{1}^*+\ldots+x_{m}^*=0$.
\qed\end{proof}

\begin{lem}\label{ll03}
Let $X$ be a normed space and $\hat\omega:=(\omega_1,\ldots,\omega_m)\in\widehat\Omega$.
Then ${N_{\widehat\Omega}(\hat\omega) =N_{\Omega_1}(\omega_1)\times\ldots\times N_{\Omega_m}(\omega_m)}$.
\end{lem}
\begin{proof}
follows directly from the definition of the Fr\'echet normal cone.
\qed\end{proof}

The proof of the main theorem of this section relies heavily on two fundamental results of variational analysis: the \emph{Ekeland variational principle} (Ekeland \cite{Eke74}; cf., e.g., \cite[Theorem~2.1]{Kru03.1}, \cite[Theorem 2.26]{Mor06.1}) and  the \emph{fuzzy} (\emph{approximate}) \emph{sum rule} (Fabian \cite{Fab89}; cf., e.g., \cite[Rule~2.2]{Kru03.1}, \cite[Theorem 2.33]{Mor06.1}).
Below we provide these results for completeness.

\begin{lem}[Ekeland variational principle] \label{l01}
Suppose $X$ is a complete metric space, and $f: X\to\R_\infty$ is lower semicontinuous and bounded from below, $\varepsilon>0, \lambda>0$. If
$$
f(v)<\inf_X f + \varepsilon,
$$
then there exists $x\in X$ such that

(a) $d(x,v)<\lambda $,

(b) $f(x)\le f(v)$,

(c) $f(u)+(\varepsilon/\lambda)d(u,x)\ge f(x)$ for all $u\in X$.
\end{lem}

\begin{lem}[Fuzzy sum rule] \label{l02}
Suppose $X$ is Asplund,
$f_1: X \to\R$ is Lipschitz continuous and
$f_2: X \to\R_\infty$
is lower semicontinuous in a neighbourhood of $\bar x$ with $f_2(\bar x)<\infty$.
Then, for any $\varepsilon>0$, there exist $x_1,x_2\in X$ with $\|x_i-\bar x\|<\varepsilon$, ${|f_i(x_i)-f_i(\bar x)|<\varepsilon}$ $(i=1,2)$ such that
$$
\partial (f_1+f_2) (\bar x) \subset \partial f_1(x_1) +\partial f_2(x_2) + \varepsilon\B^\ast.
$$
\end{lem}

The next theorem gives dual sufficient conditions for regularity of collections of sets.

\begin{thm}\label{T3.1}
Let $X$ be an Asplund space and $\Omega_1$, \ldots, $\Omega_m$ be closed.
\begin{enumerate}
\item
$\bold{\Omega}$ is subregular at $\bar x$ if there exist positive numbers $\alpha$ and $\delta$ such that, for any $\rho\in ]0,\delta[$, $x\in B_{\rho}(\bar x)$, $\omega_i\in\Omega_i\cap B_{\rho}(x)$ $(i=1,\ldots,m)$ with $\omega_i\neq x$ for some ${i\in\{1,\ldots,m\}}$, there is an $\varepsilon>0$ such that, for any $x'\in B_{\varepsilon}(x)$, $\omega'_i\in\Omega_i\cap B_{\varepsilon}(\omega_i)$, $x_{i}^*\in N_{\Omega_i}(\omega_i') +\rho\B^*$ $(i=1,\ldots,m)$
satisfying
\begin{gather}\notag
x_i^*=0
\quad\mbox{if}\quad
\norm{x'-\omega_i'}<\max_{1\le j\le m} \norm{x'-\omega_j'},
\\\notag
\langle x_i^*,x'-\omega_i'\rangle
\ge\|x_i^*\|(\|x'-\omega_i'\|-\varepsilon),
\\\label{=1}
\sum_{i=1}^m\|x_i^*\|=1,
\end{gather}
it holds
\begin{gather}\label{>al}
\left\|\sum\limits_{i=1}^m{x}_i^*\right\|>\alpha.
\end{gather}
\item
$\bold{\Omega}$ is uniformly regular at $\bar x$ if and only if there are positive numbers $\alpha$ and $\delta$ such that
\eqref{>al} holds true
for all $\omega_i\in\Omega_i\cap{B}_\delta(\bar{x})$ and
$x_i^*\in{N}_{\Omega_i}(\omega_i)$ $(i=1,\ldots,m)$ satisfying \eqref{=1}. \end{enumerate}
\end{thm}

The proof of Theorem~\ref{T3.1}~(i) consists of a series of propositions providing lower estimates for constant \eqref{errcon'} and, thus, sufficient conditions for subregularity of $\bold{\Omega}$ which can be of independent interest.
Observe that constant \eqref{errcon'} can be rewritten as
\begin{equation}\label{errcon}
\zeta[\bold{\Omega}](\bar{x})= \liminflarge_{\substack{x\to\bar{x},\, \omega_i\to\bar x\; (1\le i\le m)\\ \hat\omega=(\omega_1,\ldots,\omega_m)\\x\notin \bigcap_{i=1}^m\Omega_i}} \frac{f(x,\hat\omega)} {d\left(x,\bigcap_{i=1}^m\Omega_i\right)}
\end{equation}
with function $f:X^{m+1}\to{\R}_\infty :={\R}\cup\{+\infty\}$ defined as
\begin{equation}\label{f}
f(x,\hat x)=\max_{1\le i\le m} \norm{x-x_i} + \delta_{\widehat\Omega}(\hat x),\quad x\in X,\; \hat x:=(x_1,\ldots,x_m)\in X^m,
\end{equation}
where $\delta_{\widehat\Omega}$ is the indicator function of $\widehat\Omega$: $\delta_{\widehat\Omega}(\hat x)=0$ if $\hat x\in\widehat\Omega$ and $\delta_{\widehat\Omega}(\hat x)=+\infty$ otherwise.

\begin{prop}\label{T3.1.1}
Let $X$ be a Banach space and $\Omega_1$, \ldots, $\Omega_m$ be closed.
\begin{enumerate}
\item
$\hat\zeta[\bold{\Omega}](\bar x)\le \zeta[\bold{\Omega}](\bar{x})$, where
\begin{equation}\label{strslo}
\hat\zeta[\bold{\Omega}](\bar x):= \lim_{\rho\downarrow0} \inflarge_{\substack{\norm{x-\bar x}<\rho\\ \hat\omega=(\omega_1,\ldots,\omega_m)\in \widehat\Omega\\ 0<\max\limits_{1\le i\le m} \norm{x-\omega_i}<\rho}}
\zeta_{\rho}[\bold{\Omega}](x,\hat\omega)
\end{equation}
and, for $x\in X$ and $\hat\omega=(\omega_1,\ldots,\omega_m)\in \widehat\Omega$,
\begin{equation}\label{rhoslo}
\zeta_{\rho}[\bold{\Omega}](x,\hat\omega):= \limsup_{\substack{(u,\hat v)\to(x,\hat\omega)\\(u,\hat v)\neq (x,\hat\omega)\\ \hat v=(v_1,\ldots,v_m)\in \widehat\Omega}} \frac{\left(\max\limits_{1\le i\le m} \norm{x-\omega_i}-\max\limits_{1\le i\le m} \norm{u-v_i}\right)_+}{\norm{(u,\hat v)-(x,\hat\omega)}_{\rho}}.
\end{equation}
\item
If $\hat\zeta[\bold{\Omega}](\bar x)>0$, then $\bold{\Omega}$ is subregular at $\bar x$.
\end{enumerate}
\end{prop}

\begin{proof}
(i)
Let $\zeta[\bold{\Omega}](\bar{x})<\alpha<\infty$. Choose a $\rho\in ]0,1[$ and set
\begin{equation}\label{28}
\eta:= \min\left\{\frac{\rho}{2},\frac{\rho}{\alpha},\rho^{\frac{2}{\rho}}\right\}.
\end{equation}
By \eqref{errcon}, there are $x'\in B_{\eta}(\bar x)$ and $\hat\omega'=(\omega_1',\ldots,\omega_m')\in \widehat\Omega$ such that
\begin{equation}\label{29}
   0<f(x',\hat\omega')<\alpha d\left(x',\bigcap_{i=1}^m\Omega_i\right).
\end{equation}
Denote $\varepsilon:=f(x',\hat\omega')$ and $\mu:=d\left(x',\bigcap_{i=1}^m\Omega_i\right)$. Then $\mu\le \norm{x'-\bar x}\le \eta\le \frac{\rho}{2}<1$.
Observe that $f$ is lower semicontinuous.
Applying to $f$ Lemma~\ref{l01} with $\varepsilon$ as above and
\begin{equation}\label{30}
\lambda:=\mu(1-\mu^{\frac{\rho}{2-\rho}}),
\end{equation}
we find points $x\in X$ and $\hat\omega=(\omega_1,\ldots,\omega_m)\in X^{m}$ such that
\begin{equation}\label{31}
   \norm{(x,\hat\omega)-(x',\hat\omega')}_{\rho}<\lambda,\; f(x,\hat\omega)\le f(x',\hat\omega'),
\end{equation}
and
\begin{equation}\label{32}
f(u,\hat v)+\frac{\varepsilon}{\lambda}\norm{(u,\hat v)-(x,\hat\omega)}_{\rho} \ge f(x,\hat\omega),
\end{equation}
for all $(u,\hat v)\in X\times X^{m}$.
Thanks to \eqref{31}, \eqref{30}, \eqref{28}, and \eqref{29}, we have
$$
\norm{x-x'}<\lambda <\mu \le \norm{x'-\bar x},
$$
\begin{equation}\label{33}
   d\left(x,\bigcap_{i=1}^m\Omega_i\right)\ge d\left(x',\bigcap_{i=1}^m\Omega_i\right)- \norm{x-x'}\ge \mu-\lambda =\mu^{\frac{2}{2-\rho}},
\end{equation}
\begin{equation}\label{34}
\norm{x-\bar x}\le \norm{x-x'}+\norm{x'-\bar x}<2\norm{x'-\bar x}\le 2\eta\le\rho,
\end{equation}
\begin{equation}\label{35}
f(x,\hat\omega)\le f(x',\hat\omega')<\alpha\mu\le \alpha\eta\le\rho.
\end{equation}
It follows from \eqref{33}, \eqref{34}, and \eqref{35} that
$$
\norm{x-\bar x}<\rho,\; \hat\omega\in \widehat\Omega,\; 0<\max\limits_{1\le i\le m} \norm{x-\omega_i}<\rho.
$$
Observe that $\mu^{\frac{\rho}{2-\rho}}\le \eta^{\frac{\rho}{2-\rho}}< \eta^{\frac{\rho}{2}}\le \rho$, and consequently, by \eqref{29} and \eqref{30},
$$
\frac{\varepsilon}{\lambda} <\frac{\alpha\mu}{\lambda} =\frac{\alpha}{1-\mu^{\frac{\rho}{2-\rho}}} <\frac{\alpha}{1-\rho}.
$$
Thanks to \eqref{32} and \eqref{f}, we have
$$
\max\limits_{1\le i\le m} \norm{x-\omega_i}-\max\limits_{1\le i\le m} \norm{u-v_i}\le \frac{\alpha}{1-\rho}\norm{(u,\hat v)-(x,\hat\omega)}_{\rho}
$$
for all $u\in X$ and $\hat v=(v_1,\ldots,v_m)\in \widehat\Omega$.
It follows that $\zeta_{\rho}[\bold{\Omega}](x,\hat\omega)\le \dfrac{\alpha}{1-\rho}$ and consequently
$$
\inflarge_{\substack{\norm{x-\bar x}<\rho\\\hat\omega=(\omega_1,\ldots,\omega_m)\in \widehat\Omega\\ 0<\max\limits_{1\le i\le m} \norm{x-\omega_i}<\rho}}
\zeta_{\rho}[\bold{\Omega}](x,\hat\omega)\le \frac{\alpha}{1-\rho}.
$$
Taking limits in the last inequality as $\rho\downarrow 0$ and $\alpha\to \zeta[\bold{\Omega}](\bar{x})$ yields the claimed inequality.

(ii) follows from (i) and Proposition~\ref{theorem11}~(ii).
\qed\end{proof}

\begin{prop}\label{T3.1.2}
Let $X$ be an Asplund space and $\Omega_1$, \ldots, $\Omega_m$ be closed.
\begin{enumerate}
\item
$\hat\zeta^*_1[\bold{\Omega}](\bar x)\le \hat\zeta[\bold{\Omega}](\bar x)$, where $\hat\zeta[\bold{\Omega}](\bar x)$ is given by \eqref{strslo},
\begin{equation}\label{strsubslo'}
\hat\zeta^*_1[\bold{\Omega}](\bar x):= \lim_{\rho\downarrow 0} \inflarge_{\substack{\norm{x-\bar x}<\rho\\ \hat\omega=(\omega_1,\ldots,\omega_m)\in \widehat\Omega\\ 0<\max\limits_{1\le i\le m} \norm{x-\omega_i}<\rho}}
\zeta_{\rho,1}^*[\bold{\Omega}](x,\hat\omega)
\end{equation}
and, for $x\in X$ and $\hat\omega=(\omega_1,\ldots,\omega_m)\in \widehat\Omega$,
\begin{equation}\label{subrhoslo}
\zeta_{\rho,1}^*[\bold{\Omega}](x,\hat\omega):= \inflarge_{\substack{(x^*,\hat y^*)\in\partial f(x,\hat\omega)\\ \norm{\hat y^*}<\rho}}\norm{x^*}
\end{equation}
(with the convention that the infimum over the empty set equals $+\infty$).
\item
If $\hat\zeta^*_1[\bold{\Omega}](\bar x)>0$, then $\bold{\Omega}$ is subregular at $\bar x$.
\end{enumerate}
\end{prop}

\begin{proof}
Let $\hat\zeta[\bold{\Omega}](\bar x)< \alpha<\infty$.
Choose a $\beta\in ]\hat\zeta[\bold{\Omega}](\bar x),\alpha[$ and an arbitrary ${\rho>0}$.
Set ${\rho'=\min\{1,\alpha^{-1}\}\rho}$.
By \eqref{strslo} and \eqref{rhoslo}, one can find points $x\in X$ and ${\hat\omega=(\omega_1,\ldots,\omega_m)\in \widehat\Omega}$ such that $\norm{x-\bar x}<\rho'$, $0<\max_{1\le i\le m} \norm{\omega_i-x}<\rho'$, and
$$
\max\limits_{1\le i\le m} \norm{x-\omega_i}-\max\limits_{1\le i\le m} \norm{u-v_i}\le \beta\norm{(u,\hat v)-(x,\hat\omega)}_{\rho'}
$$
for all $(u,\hat v)$ with $\hat v=(v_1,\ldots,v_m)\in \widehat\Omega$ near $(x,\hat\omega)$.
In other words, $(x,\hat\omega)$ is a local minimizer of the function
$$
(u,\hat v)\mapsto \max\limits_{1\le i\le m} \norm{u-v_i}+\beta\norm{(u,\hat v)-(x,\hat\omega)}_{\rho'}
$$
subject to $\hat v=(v_1,\ldots,v_m)\in \widehat\Omega$.
By definition \eqref{f}, this means that $(x,\hat\omega)$ minimizes locally the function
$$
(u,\hat v)\mapsto f(u,\hat v)+\beta\norm{(u,\hat v)-(x,\hat\omega)}_{\rho'},
$$
and consequently its Fr\'echet subdifferential at $(x,\hat\omega)$ contains zero.
Take an
$$
\varepsilon\in\Bigl]0,\min\{\rho-\norm{x-\bar x},\rho-\max_{1\le i\le m} \norm{x-\omega_i},\alpha-\beta\}\Bigr[.
$$
Applying Lemma~\ref{l02} and Lemma~\ref{ll01}~(ii), we can find points $x'\in X$, ${\hat\omega'=(\omega_1',\ldots,\omega_m')\in \widehat\Omega}$, and $(x^*,\hat y^*)\in \partial f(x',\hat\omega')$ such that $$\norm{x'-x}<\varepsilon,\quad 0<\max_{1\le i\le m} \norm{x'-\omega_i'}\le \max_{1\le i\le m} \norm{x-\omega_i}+\varepsilon,$$
$$\mbox{and}\quad\norm{(x^*,\hat y^*)}_{\rho'} =\|x^*\|+\|\hat y^*\|/\rho' <\beta+\varepsilon.$$
It follows that
$$\norm{x'-\bar x}<\rho,\;\; 0<\max_{1\le i\le m} \norm{x'-\omega_i'}<\rho,\;\; \norm{x^*}<\alpha,\;\mbox{ and }\; \norm{\hat y^*}<\rho'\alpha\le \rho.$$
Hence, $\zeta_{\rho,1}^*[\bold{\Omega}](x',\hat\omega')<\alpha$, and consequently $\hat\zeta^*_1[\bold{\Omega}](\bar x)<\alpha$.
By letting $\alpha\to \hat\zeta[\bold{\Omega}](\bar x)$, we obtain the claimed inequality.

(ii) follows from (i) and Proposition~\ref{T3.1.1}~(ii).
\qed\end{proof}

\begin{prop}\label{T3.1.3}
Let $X$ be an Asplund space and $\Omega_1$, \ldots, $\Omega_m$ be closed.
\begin{enumerate}
\item
$\hat\zeta^*_2[\bold{\Omega}](\bar x)\le \hat\zeta^*_1[\bold{\Omega}](\bar x)$, where $\hat\zeta^*_1[\bold{\Omega}](\bar x)$ is given by \eqref{strsubslo'},
\begin{equation}\label{c2}
\hat\zeta^{*}_2[\bold{\Omega}](\bar x):= \lim_{\rho\downarrow 0} \inf_{\substack{\norm{x-\bar x}<\rho\\ \hat\omega=(\omega_1,\ldots,\omega_m)\in \widehat\Omega\\ 0<\max\limits_{1\le i\le m} \norm{x-\omega_i}<\rho}}
\lim_{\varepsilon\downarrow 0} \inf_{\substack{\norm{x'-x}<\varepsilon\\ \hat\omega'\in \widehat\Omega\\ \norm{\hat\omega'-\hat\omega}<\varepsilon}}
\zeta_{\rho,\varepsilon,2}^{*}[\bold{\Omega}] (x',\hat\omega')
\end{equation}
and, for $x\in X$ and $\hat\omega=(\omega_1,\ldots,\omega_m)\in \widehat\Omega$ with $(x-\omega_1,\ldots,x-\omega_m)\ne0$,
\begin{align}
\zeta_{\rho,\varepsilon,2}^{*}[\bold{\Omega}](x,\hat\omega):= \inf\Biggl\{
\left\|\sum\limits_{i=1}^m{x}_i^*\right\|:\,
&
x_{i}^*\in N_{\Omega_i}(\omega_i) +\rho\B^*\quad (i=1,\ldots,m),
\notag\\&
x_i^*=0
\quad\mbox{if}\quad
\norm{x-\omega_i}<\max_{1\le j\le m} \norm{x-\omega_j},
\notag\\&
\langle x_i^*,x-\omega_i\rangle
\ge\|x_i^*\|(\|x-\omega_i\|-\varepsilon),
\notag\\&
\sum_{i=1}^m\|x_i^*\|=1
\Biggr\}.\label{c2r}
\end{align}
\item
If $\hat\zeta^*_2[\bold{\Omega}](\bar x)>0$, then $\bold{\Omega}$ is subregular at $\bar x$.
\end{enumerate}
\end{prop}

\begin{proof}
(i)
Let $\rho>0$, $x\in X$, $\hat\omega:= (\omega_1,\ldots,\omega_m)\in\widehat\Omega$ with $\|x-\bar x\|<\rho$, ${0<\max_{1\le i\le m} \norm{x-\omega_i}<\rho}$, $(u^*,\hat v^*) \in\partial f(x,\hat\omega)$, where $f$ is given by \eqref{f}, and $\|\hat v^*\|<\rho$.
Denote $\hat v:=(x-\omega_1,\ldots,x-\omega_m)$.
Then $0<\|\hat v\|<\rho$.
Observe that function $f$ is the sum of two functions on $X^{m+1}$:
$$(x,\hat x)\mapsto \varphi(x,\hat x):= \|(x-x_1,\ldots,x-x_m)\|\quad\mbox{and}\quad(x,\hat x) \mapsto\delta_{\widehat\Omega}(\hat x),$$
where $\hat x:=(x_1,\ldots,x_m)$ and $\delta_{\widehat\Omega}$ is the indicator function of $\widehat\Omega$.
The first function is Lipschitz continuous while the second one is lower semicontinuous.
One can apply Lemma~\ref{l02}.
For any $\varepsilon>0$, there exist points $x'\in X$, $\hat x:=(x_1,\ldots,x_m)\in X^m$, $\hat\omega':=(\omega_1',\ldots,\omega_m') \in\widehat\Omega$, $\left(x^*,\hat y^*\right)\in\partial \varphi(x',\hat x)$, and $\hat\omega^*\in N_{\widehat\Omega}(\hat\omega')$ such that
\begin{gather}\notag
\|x'-x\|<\varepsilon,\quad \|\hat x-\hat\omega\|<\frac{\varepsilon}{4},\quad \|\hat\omega'-\hat\omega\|<\frac{\varepsilon}{4},
\\\label{est}
\|(u^*,\hat v^*)-(x^*,\hat y^*)-(0,\hat\omega^*)\|<\varepsilon.
\end{gather}
Taking a smaller $\varepsilon$ if necessary, one can ensure that $\hat v':=(x'-\omega_1',\ldots,x'-\omega_m')\ne0$, $\hat v'':=(x'-x_1,\ldots,x'-x_m)\ne0$, $\|\hat v^*\|+\varepsilon<\rho$
and,
for any $i=1,\ldots,m$, ${\norm{x'-x_i}<\max_{1\le j\le m} \norm{x'-x_j}}$ if and only if ${\norm{x'-\omega_i'}<\max_{1\le j\le m} \norm{x'-\omega_j'}}$.
By Lemma~\ref{ll02},
$$\hat x^*:=-\hat y^*\in J(\hat v'')
\quad\mbox{and}\quad
x^*=x_{1}^*+\ldots+x_{m}^*,$$
where
$\hat x^*=(x_1^*,\ldots,x_m^*)$.
By Proposition~\ref{prJ},
\begin{gather*}
\sum_{i=1}^m\|x_i^*\|=1,
\\
x_i^*=0
\quad\mbox{if}\quad
\norm{x'-\omega_i'}<\max_{1\le j\le m} \norm{x'-\omega_j'},
\end{gather*}
\begin{align*}
\langle x_i^*,x'-\omega_i'\rangle
&\ge\langle x_i^*,x'-x_i\rangle-\|x_i^*\|\,\|x_i-\omega_i'\| =\|x_i^*\|(\|x'-x_i\|-\|x_i-\omega_i'\|)
\\&
\ge\|x_i^*\|(\|x'-\omega_i'\|-2\|x_i-\omega_i'\|)
\ge\|x_i^*\|(\|x'-\omega_i'\|-\varepsilon) \quad(i=1,\ldots,m).
\end{align*}
Inequality \eqref{est} yields the estimates:
$\|u^*\|>\left\|{x}^*\right\|-\varepsilon$, $\left\|\hat x^*-\hat\omega^*\right\|<\|\hat v^*\|+\varepsilon<\rho$,
and consequently
$$
\|u^*\|> \left\|\sum\limits_{i=1}^m{x}_i^*\right\|-\varepsilon, \quad \hat x^*\in N_{\widehat\Omega}(\hat\omega') +\rho\B_m^*.
$$
It follows from Lemma~\ref{ll03} and definitions \eqref{subrhoslo} and \eqref{c2r} that
$$\zeta_{\rho,1}^*[\bold{\Omega}](x,\hat\omega)\ge \zeta_{\rho,\varepsilon,2}^*[\bold{\Omega}] (x',\hat\omega')-\varepsilon.$$
The claimed inequality is a consequence of the last one and definitions \eqref{strsubslo'} and \eqref{c2}.

(ii) follows from (i) and Proposition~\ref{T3.1.2}~(ii).
\qed\end{proof}

\begin{proof}\emph{of Theorem~\ref{T3.1}}
(i) follows from Proposition~\ref{T3.1.3}~(ii) and definitions \eqref{c2} and \eqref{c2r}.

(ii) is a consequence of \cite[Theorem~4]{Kru09.1}.
\qed\end{proof}

\begin{rem}
One of the main tools in the proof of Theorem~\ref{T3.1} is the fuzzy sum rule (Lemma~\ref{l02}) for Fr\'echet subdifferentials in Asplund spaces.
The statements can be extended to general Banach spaces.
For that, one has to replace Fr\'echet subdifferentials (and normal cones) with some other kind of subdifferentials satisfying a certain set of natural properties including the sum rule (not necessarily fuzzy) -- cf. \cite[p. 345]{KruLop12.1}.

If the sets $\Omega_1$, \ldots $\Omega_m$ are convex or the norm of $X$ is Fr\'echet differentiable away from $0$, then the fuzzy sum rule can be replaced in the proof by either the convex sum rule (Moreau--Rockafellar formula) or the simple (exact) differentiable rule (see, e.g., \cite[Corollary~1.12.2]{Kru03.1}), respectively, to produce dual sufficient conditions for regularity of collections of sets in general Banach spaces in terms of either normals in the sense of convex analysis or Fr\'echet normals.
\end{rem}
\begin{rem}
Since uniform regularity is a stronger property than subregularity (Remark~\ref{rem1}), the criterion in part (ii) of Theorem~\ref{T3.1} is also sufficient for the subregularity of the collection of sets in part (i).
\end{rem}

The next example illustrates application of Theorem~\ref{T3.1}~(i) for detecting subregularity of collections of sets.

\begin{examp}
Consider the collection $\{\Omega,\Omega\}$ of two copies of the set
${\Omega:=\R\times\{0\}}$
in the real plane ${\R}^2$ with the Euclidean norm (cf. Example~\ref{exa1})
and the point ${\bar x = (0,0) \in \Omega}$.

Obviously $N_{\Omega}(\omega)=\{0\}\times\R$ for any $\omega\in\Omega$.
If $x_1^*:=(a_1,b_1)\in N_{\Omega}(\omega'_1)+\rho\B^*$ and $x_2^*:=(a_2,b_2)\in N_{\Omega}(\omega'_2)+\rho\B^*$ for some $\omega'_1,\omega'_2\in\Omega$, then $|a_1|\le\rho$ and $|a_2|\le\rho$.

Take any positive numbers $\alpha$ and $\delta$ such that $\alpha^2+2\delta^2<1$ and any $\rho\in]0,\delta[$.
Let $\omega_1,\omega_2\in\Omega$, $x\in\R^2$, $\hat v:=(\omega_1-x,\omega_2-x)\in\R^4\setminus\{0\}$.
Because of the definition of $\Omega$, $\hat v$ has the following representation: $\hat v=(v_1,v,v_3,v)$.

If $v=0$, then $\xi:=v_1^2+v_3^2>0$.
Choose an $\varepsilon>0$ such that $$(\max\{|v_1|-\varepsilon,0\})^2 +(\max\{|v_3|-\varepsilon,0\})^2>\xi/2
\quad\mbox{and}\quad 4\varepsilon^2/\xi<\alpha^2.$$
There are no pairs $x_{1}^*$, $x_{2}^*$ satisfying the conditions of Theorem~\ref{T3.1}~(i).
Indeed, if ${\hat v':=(v_1',v_2',v_3',v_4')\in B_\varepsilon(\hat v)}$, then $|v_2'|\le\varepsilon$, $|v_4'|\le\varepsilon$, and $\|\hat v'\|^2\ge|v_1'|^2+|v_3'|^2>\xi/2$.
If $(x_{1}^*,x_{2}^*)\in J(\hat v')$, then  ${(x_{1}^*,x_{2}^*)=\hat v'/\|\hat v'\|}$.
Hence, $b_1^2+b_2^2\le2\varepsilon^2/\|\hat v'\|^2<4\varepsilon^2/\xi<\alpha^2$ and consequently $\|(x_{1}^*,x_{2}^*)\|<\alpha^2+2\delta^2<1$; a contradiction.

If $v\ne0$, then we choose an $\varepsilon\in(0,|v|)$.
If $\hat v'\in B_\varepsilon(\hat v)$ and $(x_{1}^*,x_{2}^*)\in J(\hat v')$, then $b_1$ and $b_2$ have the same sign as $v$ and $b_1^2+b_2^2\ge1-2\delta^2$.
Hence,
$$
\|x_{1}^*+x_{2}^*\|^2=(a_1+a_2)^2+(b_1+b_2)^2\ge(b_1+b_2)^2\ge b_1^2+b_2^2>\alpha^2.
$$
By Theorem~\ref{T3.1}~(i), the collection $\{\Omega,\Omega\}$ is subregular at $\bar x$.
\end{examp}

\section{Regularity of Set-Valued Mappings} \label{S5}

In this section, we present relationships between regularity properties of collections of sets and the corresponding properties of set-valued mappings, which have been intensively investigated; cf., e.g., \cite{RocWet98,Mor06.1,Iof00_,DonLewRoc03,DonRoc04,DonRoc09,Kru09.1,Pen89}.

Consider a set-valued mapping $F:X\rightrightarrows Y$ between metric spaces and a point $(\bar x,\bar y)\in \gr F:=\{(x,y)\in X\times Y:\; y\in F(x)\}$.
\begin{dfn}\label{metricR}
\begin{enumerate}
\item
$F$ is \emph{metrically semiregular} at $(\bar{x},\bar{y})$ iff there exist positive numbers $\gamma$ and $\delta$ such that
\begin{equation}\label{mhr}
    \gamma d\left(\bar x,F^{-1}(y)\right) \le d(y,\bar y),\; \forall y \in B_{\delta}(\bar{y}).
\end{equation}
The exact upper bound of all numbers $\gamma$ such that (\ref{mhr}) is satisfied will be denoted by $\theta[F](\bar{x},\bar y)$.
\item
$F$ is \emph{metrically subregular} at $(\bar{x},\bar{y})$ iff there exist positive numbers $\gamma$ and $\delta$ such that
\begin{equation}\label{msr}
    \gamma d\left(x,F^{-1}(\bar y)\right) \le d(\bar y,F(x)),\; \forall x \in B_{\delta}(\bar{x}).
\end{equation}
The exact upper bound of all numbers $\gamma$ such that (\ref{msr}) is satisfied will be denoted by $\zeta[F](\bar{x},\bar y)$.
\item
$F$ is \emph{metrically regular} at $(\bar{x},\bar{y})$ iff there exist positive numbers $\gamma$ and $\delta$ such that
\begin{equation}\label{mr}
    \gamma d\left(x,F^{-1}(y)\right) \le d\left(y,F(x)\right) ,\; \forall (x,y) \in B_{\delta}(\bar{x},\bar{y}).
\end{equation}
The exact upper bound of all numbers $\gamma$ such that (\ref{mr}) is satisfied will be denoted by $\hat{\theta}[F](\bar{x},\bar y)$.
\end{enumerate}
\end{dfn}
\begin{rem}
Property (ii) and especially property (iii) in Definition \ref{metricR} are very well known and widely used in variational analysis; see, e.g.,
\cite{RocWet98,Mor06.1,Iof00_,DonRoc09,Kru09.1,Pen89,DonLewRoc03,DonRoc04,ZheNg07,ZheNg10,ZheNg12}. Property (i)  was introduced in \cite{Kru09.1}.
In \cite{ArtMor11,ApeDurStr13}, it is referred to as \emph{metric hemiregularity}.
\end{rem}

For a collection of sets $\bold{\Omega}:=\{\Omega_1,\ldots,\Omega_m\}$ in a normed linear space $X$, one can consider set-valued mapping $F:X \rightrightarrows X^m$ defined by (cf. \cite[Proposition~5]{Iof00_}, \cite[Theorem 3]{Kru05.1}, \cite[Proposition 8]{Kru06.1}, \cite[p. 491]{LewLukMal09}, \cite[Proposition 33]{HesLuk})
$$
F(x):= (\Omega_1 -x)\times\ldots\times (\Omega_m -x),\; \forall x \in X.
$$
It is easy to check that, for $x \in X$ and $u=(u_1,\ldots,u_m) \in X^m$, it holds
$$
x \in \bigcap_{i=1}^m\Omega_i \iff 0 \in F(x),\quad F^{-1}(u) = \bigcap_{i=1}^m (\Omega_i-u_i).
$$

The next proposition is a consequence of Theorem~\ref{T2.1}.

\begin{prop}\label{theorem13}
Consider $\bold{\Omega}$ and $F$ as above and a point $\bar x\in \bigcap_{i=1}^m\Omega_i$.
\begin{enumerate}
\item
$\bold{\Omega}$ is semiregular at $\bar x$ if and only if $F$ is metrically semiregular at $(\bar x,0)$.
Moreover, $\theta[\bold{\Omega}](\bar{x})=\theta[F](\bar{x},0)$.
\item
$\bold{\Omega}$ is subregular at $\bar x$ if and only if $F$ is metrically subregular at $(\bar x,0)$.
Moreover, $\zeta[\bold{\Omega}](\bar{x})=\zeta[F](\bar{x},0)$.
\item
$\bold{\Omega}$ is uniformly regular at $\bar x$ if and only if $F$ is metrically regular at $(\bar x,0)$.
Moreover, $\hat{\theta}[\bold{\Omega}](\bar{x})=\hat{\theta}[F](\bar{x},0)$.
\end{enumerate}
\end{prop}

\begin{rem}
Assertion (iii) was proved in \cite[Proposition 8]{Kru06.1} (see also \cite[Theorem 3]{Kru05.1} and \cite[p. 491]{LewLukMal09}).
The equivalence of subregularity of $\bold{\Omega}$ and metric subregularity of $F$ has been established by Hesse and Luke in Proposition 33 (ii) of their recent preprint \cite{HesLuk}.
This proposition has not been included in the final version of their article which appeared in \cite{HesLuk13}.
\end{rem}

Conversely, regularity properties of set-valued mappings between normed linear spaces can be treated as realizations of the corresponding regularity properties of certain collections of two sets.

For a given set-valued mapping $F:X\rightrightarrows Y$ between normed linear spaces and a point $(\bar x,\bar y)\in \gr F$, one can consider the collection $\bold{\Omega}$ of two sets $\Omega_1= \gr F$ and $\Omega_2=X\times \{\bar y\}$ in $X\times Y$.
It is obvious that $(\bar x,\bar y)\in \Omega_1\cap\Omega_2$.

\begin{thm}\label{theorem14}
Consider $F$ and $\bold{\Omega}$ as above.
\begin{enumerate}
\item
$F$ is metrically semiregular at $(\bar x,\bar y)$ if and only if $\bold{\Omega}$ is semiregular at $(\bar x,\bar y)$.
Moreover,
\begin{equation}\label{newestSem}
\frac{\theta[F](\bar{x},\bar y)}{\theta[F](\bar{x},\bar y)+2}\le \theta[\bold{\Omega}](\bar{x},\bar y)\le \theta[F](\bar{x},\bar y)/2.
\end{equation}
\item
$F$ is metrically subregular at $(\bar x,\bar y)$ if and only if $\bold{\Omega}$ is subregular at $(\bar x,\bar y)$.
Moreover,
\begin{equation}\label{newest}
\frac{\zeta[F](\bar{x},\bar y)}{\zeta[F](\bar{x},\bar y)+2}\le \zeta[\bold{\Omega}](\bar{x},\bar y)\le \min\{\zeta[F](\bar{x},\bar y)/2,1\}.
\end{equation}
\item
$F$ is metrically regular at $(\bar x,\bar y)$ if and only if $\bold{\Omega}$ is uniformly regular at $(\bar x,\bar y)$.
Moreover,
\begin{equation}\label{newestReg}
\frac{\hat{\theta}[F](\bar{x},\bar y)}{\hat{\theta}[F](\bar{x},\bar y)+2}\le \hat{\theta}[\bold{\Omega}](\bar{x},\bar y)\le \min\{\hat{\theta}[F](\bar{x},\bar y)/2,1\}.
\end{equation}
\end{enumerate}
\end{thm}
\begin{proof} (i) Suppose $F$ is metrically semiregular at $(\bar x,\bar y)$, i.e., $\theta[F](\bar{x},\bar y)>0$. Fix a $\gamma\in ]0,\theta[F](\bar{x},\bar y)[$. Then there exists a number $\delta'>0$ such that \eqref{mhr} is satisfied for all $y\in B_{\delta'}(\bar y)$. Take any $\alpha>0$ satisfying $2\alpha/\gamma+\alpha<1$, and a $\delta:=\frac{\delta'}{2\alpha}$. We are going to check that
\begin{equation}\label{HR'}
\left(\Omega_1-(u_1,v_1)\right)\bigcap \left(\Omega_2-(u_2,v_2)\right)\bigcap B_{\rho}(\bar x,\bar y)\neq \emptyset
\end{equation}
for all $\rho\in ]0,\delta[$ and $(u_1,v_1)$, $(u_2,v_2)\in (\alpha\rho)\B $.
Indeed, take any $\rho\in ]0,\delta[$ and $(u_1,v_1),(u_2,v_2)\in (\alpha\rho)\B $. We need to find a point $(x,y)\in B_{\rho}(\bar x,\bar y)$ satisfying
\sloppy
$$
\left\{
  \begin{array}{ll}
    (x,y)+(u_1,v_1)\in \gr F,\\
    y=\bar y-v_2\hbox{.}
  \end{array}
\right.
$$
We set $y':=\bar y-v_2+v_1$, so $y' \in B_{\delta'}(\bar y)$ as $\|y'-\bar y\|=\|v_1-v_2\|\le 2\alpha\rho<2\alpha\delta=\delta'$.
Then there is, by \eqref{mhr}, an $x'\in F^{-1}(y')$ such that
$$
\|\bar x-x'\|\le \frac{1}{\gamma}\|\bar y-y'\|.
$$
Put $y:=y'-v_1=\bar y-v_2$ and $x:=x'-u_1$.
Then it holds
$$(x,y)+(u_1,v_1)=(x',y') \in \gr F,\quad \norm{y-\bar y}=\|v_2\|\le \alpha\rho<\rho,$$
and
\begin{multline*}
\|x-\bar x\|\le \|x-x'\|+\|x'-\bar x\|\le \|u_1\|+\frac{1}{\gamma}\|\bar y-y'\|\\
=\|u_1\|+\frac{1}{\gamma}\|v_1-v_2\| \le (2\alpha/\gamma+\alpha)\rho<\rho.
\end{multline*}
Hence, \eqref{HR'} is proved.

The above reasoning also yields the first inequality in \eqref{newestSem}.

To prove the inverse implication, we suppose $\bold{\Omega}$ is semiregular at $(\bar x,\bar y)$, i.e., $\theta[\bold{\Omega}](\bar{x},\bar y)>0$.
Fix an $\alpha\in ]0,\theta[\bold{\Omega}](\bar{x},\bar y)[$. Then there exists a $\delta'>0$ such that \eqref{HR'} holds true for all $\rho\in ]0,\delta'[$ and $(u_1,v_1),(u_2,v_2)\in (\alpha\rho)\B $.
Set $\gamma:=2\alpha$ and $\delta<\alpha\delta'$. We are going to check that \eqref{mhr} is satisfied.
Take any $y\in B_{\delta}(\bar y)$, i.e., $\|y-\bar y\|\le \delta<\alpha\delta'$. Set $r\in ]0,\delta'[$ such that $\|y-\bar y\|=\alpha r$.
Then, applying \eqref{HR'} for $(u_1,v_1):=\left(0,\frac{y-\bar y}{2}\right),(u_2,v_2)=:\left(0,\frac{\bar y-y}{2}\right)\in \left(\alpha\frac{r}{2}\right)\B $, we can find
$(x_1,y_1)\in \gr F$ and $(x_2,\bar y)\in \Omega_2$ satisfying
$$
(x_1,y_1)-(u_1,v_1)=(x_2,\bar y)-(u_2,v_2)\in B_{\frac{r}{2}}(\bar x,\bar y).
$$
This implies that $y_1=y$, $x_1\in F^{-1}(y)$ and
$$
\|x_1-\bar x\|\le \frac{r}{2}=\frac{1}{2\alpha}\|y-\bar y\|=\frac{1}{\gamma}\|y-\bar y\|.
$$
Hence, \eqref{mhr} holds true.

The last reasoning also yields the second inequality in \eqref{newestSem}.

(ii) Suppose $F$ is metrically subregular at $(\bar x,\bar y)$, i.e., $\zeta[F](\bar{x},\bar y)>0$.
Fix a ${\gamma\in ]0,\zeta[F](\bar{x},\bar y)[}$. Then there exists $\delta'>0$ such that \eqref{msr} is satisfied for all $x\in B_{\delta'}(\bar x)$.
Take an $\alpha>0$ satisfying $2\alpha/\gamma+\alpha<1$, and a $\delta:=\frac{\delta'}{\alpha+1}$. We are going to check that
\begin{equation}\label{Hlr''}
\left(\Omega_1+ (\alpha\rho)\B \right)\bigcap \left(\Omega_2+ (\alpha\rho)\B \right)\bigcap B_{\delta}(\bar x,\bar y)\subseteq \Omega_1\cap \Omega_2+\rho\B
\end{equation}
for all $\rho\in ]0,\delta[$.
Indeed, take any
$$
(x,y)\in\left(\Omega_1+ (\alpha\rho)\B \right)\bigcap \left(\Omega_2+ (\alpha\rho)\B \right)\bigcap B_{\delta}(\bar x,\bar y).
$$
Then
$(x,y)=(x_1,y_1)+(u_1,v_1)=(x_2,\bar y)+(u_2,v_2)$ for some $(x_1,y_1)\in \gr F$, $x_2\in X$, and $(u_1,v_1)$, $(u_2,v_2)\in (\alpha\rho)\B $.
Since
$$
\|x_1-\bar x\|\le\|u_1\|+\|x-\bar x\| \le \alpha\rho+\delta<(\alpha+1)\delta=\delta',
$$
by \eqref{msr}, there exists an $x'\in F^{-1}(\bar y)$ such that $\|x_1-x'\|\le \frac{1}{\gamma}d(\bar y,F(x_1))\le \frac{1}{\gamma}\|\bar y-y_1\|$. Then
\begin{align*}
\norm{x_1-x'+u_1}\le
&\frac{1}{\gamma}\|\bar y-y_1\|+\norm{u_1}=\frac{1}{\gamma}\norm{v_1-v_2}+\norm{u_1}
\\
\le &\frac{2\alpha\rho}{\gamma}+\alpha\rho = \Big{(}\frac{2}{\gamma}+1\Big{)}\alpha\rho<\rho,
\\
\norm{v_2}\le & \alpha\rho<\rho.
\end{align*}
Hence, $(x,y)=(x',\bar y)+(x_1-x'+u_1,v_2)\in \Omega_1\cap\Omega_2+\rho\B $.

The above reasoning also yields the first inequality in \eqref{newest}.

To prove the inverse implication, we suppose that $\bold{\Omega}$ is subregular at $(\bar x,\bar y)$, i.e., $\zeta[\bold{\Omega}](\bar{x},\bar y)>0$. Fix an $\alpha\in ]0,\zeta[\bold{\Omega}](\bar{x},\bar y)[$. Then there exists a $\delta'>0$ such that \eqref{Hlr''} holds true for all $\rho\in ]0,\delta'[$.
Set $\gamma:=2\alpha>0$ and $\delta:=\min\left\{\delta',\gamma\delta',\frac{2\delta'}{\gamma}\right\}$. We are going to check that \eqref{msr} holds true.
Take any $x\in B_{\delta}(\bar x)$.
Because $d(x,F^{-1}(\bar y))\le \|x-\bar x\|\le \delta$, it is sufficient to consider the case $0<d(\bar y,F(x))<\gamma\delta$. We choose a $y\in F(x)$ such that $d(\bar y, F(x))\le \|y-\bar y\|:=r <\gamma\delta$. Then
$$
\left(x,\frac{y+\bar y}{2}\right)=(x,y)+\left(0,\frac{\bar y-y}{2}\right)=(x,\bar y)+\left(0,\frac{y-\bar y}{2}\right),\quad\norm{\frac{\bar y-y}{2}}= \frac{ r}{2}<\delta',
$$
and consequently
\begin{equation}\label{e43}
\left(x,\frac{y+\bar y}{2}\right)\in \left(\Omega_1+ \frac{ r}{2}\B \right)\bigcap \left(\Omega_2+ \frac{ r}{2}\B \right)\bigcap B_{\delta'}(\bar x,\bar y).
\end{equation}
Take $\rho:=\frac{r}{2\alpha}< \delta\le \delta'$.
Then $\frac{r}{2}=\alpha\rho$, and it follows from \eqref{Hlr''} and \eqref{e43} that
$$
\left(x,\frac{y+\bar y}{2}\right)\in \Omega_1\cap \Omega_2+\frac{r}{2\alpha}\B  =F^{-1}(\bar y)\times\{\bar y\} +\frac{\|y-\bar y\|}{\gamma}\B .
$$
Hence, there is an $x'\in F^{-1}(\bar y)$ such that
$$
\|x-x'\|\le \frac{1}{\gamma}\|y-\bar y\|.
$$
Taking infimum in the last inequality over $x'\in F^{-1}(\bar y)$ and $y\in F(x)$, we arrive at \eqref{msr}.

The last reasoning together with $\zeta[\bold{\Omega}](\bar{x},\bar y)\le 1$, in view of \eqref{errcon'}, yields the second inequality in \eqref{newest}.

(iii) Suppose $F$ is metrically regular at $(\bar x,\bar y)$, i.e., $\hat{\theta}[F](\bar{x},\bar y)>0$.
Fix a ${\gamma\in ]0,\hat{\theta}[F](\bar{x},\bar y)[}$.
Then there exists a $\delta'>0$ such that \eqref{mr} is satisfied for all $(x,y)\in B_{\delta'}(\bar x,\bar y)$.
Take an $\alpha>0$ satisfying $2\alpha/\gamma+\alpha<1$, and a $\delta:=\frac{\delta'}{2\alpha+1}$. We are going to check that
\begin{equation}\label{HUR'}
\left(\Omega_1-(x_1,y_1)-(u_1,v_1)\right)\bigcap \left(\Omega_2-(x_2,\bar y)-(u_2,v_2)\right)\bigcap (\rho\B )\neq \emptyset
\end{equation}
for all $\rho\in ]0,\delta[$, $(x_1,y_1)\in \Omega_1\cap B_{\delta}(\bar x,\bar y),x_2\in B_{\delta}(\bar x)$, and $(u_1,v_1),(u_2,v_2)\in (\alpha\rho)\B $.
Take any such $\rho,(x_1,y_1),x_2,(u_1,v_1)$, and $(u_2,v_2)$. We need to find $(a,b)\in \rho\B $ satisfying
$$
\left\{
  \begin{array}{ll}
    (x_1,y_1)+(u_1,v_1)+(a,b)\in \gr F,\\
    b=-v_2\hbox{.}
  \end{array}
\right.
$$
We set $y'=y_1-v_2+v_1$, so $y' \in B_{\delta'}(\bar y)$ as
$$
\|y'-\bar y\|\leq \|y'-y_1\|+\|y_1-\bar y\|\le \|v_1-v_2\|+\delta \le 2\alpha\rho+\delta<(2\alpha+1)\delta=\delta'.
$$
Then, applying \eqref{mr} for $(x_1,y')\in B_{\delta'}(\bar x,\bar y)$, we find $x'\in F^{-1}(y')$ such that
$$
\|x_1-x'\|\le \frac{1}{\gamma}d(y',F(x_1))\le\frac{1}{\gamma}\|y'-y_1\|=\frac{1}{\gamma}\|v_1-v_2\|\le \frac{2\alpha\rho}{\gamma}.
$$
Put $a=x'-x_1-u_1$ and $b=-v_2$. Then $\|a\|\le \|x'-x_1\|+\|u_1\|\le (2\alpha/\gamma+\alpha)\rho<\rho$, $\|b\|\le \alpha\rho<\rho$, and it holds
$(x_1,y_1)+(u_1,v_1)+(a,b)=(x',y') \in \gr F$.\\
Hence, \eqref{HUR'} is proved.

The above reasoning also yields the first inequality in \eqref{newestReg}.

To prove the inverse implication, we suppose that $\bold{\Omega}$ is uniformly regular at $(\bar x,\bar y)$, i.e., $\hat{\theta}[\bold{\Omega}](\bar{x},\bar y)>0$. Fix an $\alpha\in ]0,\hat{\theta}[\bold{\Omega}](\bar{x},\bar y)[$. Then there exists a $\delta'>0$ such that \eqref{HUR'} holds true for all $\rho\in ]0,\delta'[$, $(x_1,y_1)\in \Omega_1\cap B_{\delta'}(\bar x,\bar y),x_2\in B_{\delta'}(\bar x)$, and $(u_1,v_1)$, ${(u_2,v_2)\in (\alpha\rho)\B}$.
Set $\gamma:=2\alpha>0$. Because $\theta[\bold{\Omega}](\bar{x},\bar y)\ge \hat{\theta}[\bold{\Omega}](\bar{x},\bar y)$ (see Remark \ref{newrem}), assertion (i) implies that there exists a $\delta^*>0$ such that \eqref{mhr} is satisfied for all ${y\in B_{\delta^*}(\bar y)}$.
Set
\begin{equation}\label{del}
\delta:=\min\left\{\delta^*,\frac{\delta'}{2\alpha+2},\frac{\alpha\delta'}{2\alpha+1}\right\}>0.
\end{equation}
Now take any $(x,y)\in B_{\delta}(\bar x,\bar y)$.
We are going to check that \eqref{mr} is satisfied. Because \eqref{mhr} implies
\begin{equation*}
    \gamma d(x,F^{-1}(y))\le \gamma\|x-\bar x\|+\gamma d(\bar x,F^{-1}(y))\le \gamma\delta+\|y-\bar y\|\le (\gamma+1)\delta,
\end{equation*}
it suffices to consider the case $d(y,F(x))<(\gamma+1)\delta\le \alpha\delta'$.
Choose a $y'\in F(x)$ such that
$$d(y,F(x))\le \|y-y'\|<(\gamma+1)\delta,$$
and set $r\in ]0,\delta'[$ such that $\|y-y'\|=\alpha r<\alpha\delta'$.
Then
$$\|y'-\bar y\|\le \|y'-y\|+\|y-\bar y\|< (2\alpha+2)\delta\le \delta'$$
due to \eqref{del}.
Applying \eqref{HUR'} with
$$(x_1,y_1):=(x,y')\in \gr F\cap B_{\delta'}(\bar x,\bar y),\quad (x_2,y_2):=(\bar x,\bar y),$$ $$(u_1,v_1):=\left(0,\frac{y-y'}{2}\right),\quad (u_2,v_2)=:\left(0,\frac{y'-y}{2}\right)\in \left(\alpha\frac{r}{2}\right)\B,$$
we can find
$(\tilde x,\tilde y)\in \gr F$ and $(z,\bar y)\in \Omega_2$ satisfying
$$
(\tilde x,\tilde y)-(x_1,y_1)-(u_1,v_1)=(z,\bar y)-(x_2,\bar y)-(u_2,v_2)\in \frac{r}{2}\B .
$$
This implies $\tilde x-x_1\in \frac{r}{2}\B $ and $\tilde y=y_1+v_1-v_2=y$, so $\tilde x\in F^{-1}(y)$. Then we obtain
$$
d(x,F^{-1}(y))\le \|x-\tilde x\|\le \frac{r}{2}=\frac{1}{2\alpha}\|y-y'\|=\frac{1}{\gamma}\|y-y'\|.
$$
Taking infimum in the last inequality over $y'\in F(x)$, we arrive at \eqref{mr}.

The last reasoning together with $\hat{\theta}[\bold{\Omega}](\bar{x},\bar y)\le 1$, in view of \eqref{vartheta}, yields the second inequality in \eqref{newestReg}.
\qed
\end{proof}
\begin{rem}
The equivalences stated in Theorem \ref{theorem14} (i) and (iii) has been proved in \cite[Theorem 7]{Kru09.1} by using some auxiliary set-valued mapping.
The first inequalities in \eqref{newestSem} and \eqref{newestReg} improve the corresponding estimates given in the aforementioned reference
because it is always true that
\begin{gather*}
\frac{1}{2}\min\{\theta[F](\bar{x},\bar y)/2,1\}\le \frac{\theta[F](\bar{x},\bar y)}{\theta[F](\bar{x},\bar y)+2},
\\
\frac{1}{2}\min\{\hat\theta[F](\bar{x},\bar y)/2,1\}\le \frac{\hat\theta[F](\bar{x},\bar y)}{\hat\theta[F](\bar{x},\bar y)+2}.
\end{gather*}
Statement (ii) in Theorem \ref{theorem14} seems to be new.
\end{rem}

\section{Conclusions}

In this article, we continue investigating regularity properties of collections of sets in normed linear spaces.

We systematically examine three closely related primal space local regularity properties: \emph{semiregularity}, \emph{subregularity}, and \emph{uniform regularity} and their quantitative characterizations.
In Theorem~\ref{T2.1}, we establish equivalent metric characterizations of the three mentioned properties and demonstrate, in particular, the equivalence of subregularity and another important property, usually referred to as local \emph{linear regularity}.

In Theorem~\ref{T3.1}~(i), in the Asplund space setting, we give a new dual space sufficient condition of subregularity in terms of Fr\'echet normals.
The proof of this theorem consists of a series of propositions providing other (primal and dual space) sufficient conditions of subregularity which can be of independent interest.

We present also relationships between the mentioned regularity properties of collections of sets and the corresponding regularity properties of set-valued mappings which, in particular, explain the terminology adopted in this article.

The definitions and characterizations of the regularity properties of collections of sets discussed in this article can be extended to the more general H\"older type setting -- cf. \cite{KruTha}.

\begin{acknowledgement}
The authors wish to thank the referees for very careful reading of the paper and many valuable comments and suggestions, which helped us improve the presentation.
We also thank Editor-in-Chief Professor Franco Giannessi for his time and patience when handling our paper.
\end{acknowledgement}

\end{document}